\documentclass[12pt]{amsart}
\usepackage{amsmath,amsthm,amsfonts,amssymb,latexsym,verbatim,bbm,longtable}
\input xy
\xyoption{all}
\usepackage{hyperref,multirow}
\usepackage[shortlabels]{enumitem}
\usepackage{young}
\usepackage{array}
\usepackage{color}
\usepackage{tikz}
\usetikzlibrary{matrix,arrows}

\allowdisplaybreaks[2] \textwidth15.1cm \textheight22cm \headheight12pt \oddsidemargin.4cm
\theoremstyle{plain}
\newtheorem{theorem}{Theorem}[section]

\newtheorem{lemma}[theorem]{Lemma}

\newtheorem{prop}[theorem]{Proposition}

\newtheorem{corollary}[theorem]{Corollary}

\newcommand{\iso}{\cong}
\newcommand{\arr}{\rightarrow}

\newcommand{\incl}{\hookrightarrow}

\DeclareMathOperator{\supp}{Supp}

\newcommand{\C}{\mathbb{C}}
\newcommand{\Z}{\mathbb{Z}}

\newcommand{\mcB}{\mathcal{B}}

\newcommand{\mcG}{\mathcal{G}}
\newcommand{\mcX}{\mathcal{X}}

\newcommand{\mcP}{\mathcal{P}}

\newcommand{\mfg}{\mathfrak{g}}
\newcommand{\mfu}{\mathfrak{u}}
\newcommand{\mfp}{\mathfrak{p}}
\newcommand{\mfh}{\mathfrak{h}}

\DeclareMathOperator{\Lie}{Lie}
\DeclareMathOperator{\SO}{SO}
\DeclareMathOperator{\Sp}{Sp}

\begin{document}

\title{The cotangent bundle of a cominuscule Grassmannian}

\author{V. Lakshmibai}
\email{lakshmibai@neu.edu}

\author{Vijay Ravikumar}
\email{vijayr@cmi.ac.in}

\author{William Slofstra}
\email{wslofstra@math.ucdavis.edu}

\begin{abstract}
    A theorem of the first author states that the cotangent bundle of the type
    $A$ Grassmannian variety can be embedded as an open subset of a smooth
    Schubert variety in a two-step affine partial flag variety. We extend this
    result to cotangent bundles of cominuscule generalized Grassmannians of
    arbitrary Lie type.
\end{abstract}

\maketitle

\section{Introduction}

Earlier work of Lusztig and Strickland suggests possible connections between
the conormal varieties to partial flag varieties on the one hand, and affine
Schubert varieties on the other.  In particular, Lusztig relates certain orbit
closures arising from the type $A$ cyclic quiver $\widehat{A}_h$ to affine
Schubert varieties \cite{Lu90}. In the case $h=2$, Strickland relates such
orbit closures to conormal varieties of determinantal varieties \cite{St82};
furthermore, any determinantal variety can be canonically realized as an open
subset of a Schubert variety in the Grassmannian \cite{LS78}.

Inspired by these results, the first author was interested in finding a
relationship between affine Schubert varieties and conormal varieties to the
Grassmannian.  As a first step, she showed that the compactification of the
cotangent bundle to the Grassmannian is canonically isomorphic to a Schubert
variety in a two-step affine partial flag variety \cite{La14}. In this paper we
extend her result to cominuscule generalized Grassmannians of arbitrary finite
type (such Grassmannians occur in types $A-E$).

\subsection{Preliminaries}
Let $G_0$ be a simple algebraic group over $\C$ with associated Lie algebra
$\mfg_0$  and simple roots $\{\alpha_1, \ldots, \alpha_n\}$.  A simple root
$\alpha_i$ is \emph{cominuscule} if the coefficient of $\alpha_i$ in any
positive root of $\mfg_0$ (written in the simple root basis) is less than or
equal to $1$.  

The Weyl group of $G_0$ is generated by simple reflections $S_0 := \{s_1,
\ldots, s_n\}$ corresponding to the the simple roots $\{\alpha_1, \ldots,
\alpha_n\}$.  For any subset $K \subset S_0$, we let $P_K \subset G_0$ denote
the parabolic subgroup whose Weyl group is generated by the elements of $K$.
For $1 \leq i \leq n$, set $S_{0,i} := S_0 \setminus \{s_i\}$, so that
$P_{S_{0,i}}$ is a maximal parabolic subgroup of $G_0$.  The manifold $G_0 /
P_{S_{0,i}}$ is called a \emph{generalized Grassmannian of type $G_0$}, and is
said to be \emph{cominuscule} if $\alpha_i$ is cominuscule. For the remainder
of the paper, we fix $m \in [1,n]$ and consider the generalized Grassmannian $X
:= G_0 / P_{S_{0,m}}$ associated to $\alpha_m$. Note that $\alpha_m$ may or
may not be cominuscule at this point.

Let $\mfg$ denote the affine untwisted Kac-Moody algebra associated to
$\mfg_0$, and let $\mcG$ be the corresponding affine Kac-Moody group (see
\cite[\S 6]{Ku02}).\footnote{We use calligraphic font (e.g. $\mcG$ and
$\mcP_J$) for infinite-dimensional Kac-Moody groups, and non-calligraphic
font for finite-dimensional Lie groups (e.g. $G_0$, $P_J$).}
The Dynkin diagram for $\mfg$ depends on the Dynkin diagram for $\mfg_0$, and
is shown in Table \ref{TBL:comin} (see \cite[\S 18.1]{Ca05} or \cite[\S 4.8]{Ka90}).  We use the
convention that the affine node (sometimes called the special node) is labelled by zero, 
and similarly let $\alpha_0$ and $s_0$ be the affine simple root and reflection respectively.
The Weyl group $W$ of $\mfg$ is generated by $S := \{s_0, \ldots, s_n\}$, and
there is a parahoric subgroup $\mcP_K \subset \mcG$ associated to any subset $K
\subset S$. We let $\mcX_K := \mcG / \mcP_K$ denote the associated affine flag
variety, and $W_K \subset W$ denote the Weyl group of $\mcP_K$, or in other
words the subgroup of $W$ generated by $K$.  For any subsets $I \subset K
\subset S$, let $W^I_K \subset W_K$ denote the set of minimal length coset
representatives of $W_K / W_I$.  In particular, $W^K := W^K_S$ is the set of
minimal length coset representatives of $W / W_K$, and elements $w \in W_K$
index Schubert varieties $\mcX_K(w)$ of $\mcX_K$.  

Observe that $S_0 = S \setminus \{s_0\}$.  Let $S_m := S \setminus \{s_m\}$ and $J
:= S_{0,m} = S \setminus \{s_0,s_m\}$.  Let $w_i$ be the maximal element of
$W^J_{S_i}$, where $i \in \{0,m\}$. It is a standard fact
that $\mcX_{S_m}(w_0) \iso \mcX_J(w_0) \iso X$ (see Lemma \ref{L:isomorphic_copies_of_X}).
The basis of this note is the following
elementary but crucial observation:
\begin{lemma}\label{L:comin}
    If $\alpha_m$ is cominuscule in $\mfg_0$ then $\mcX_J(w_0)$ and $\mcX_J(w_m)$
    are isomorphic.
\end{lemma}
\begin{proof}
    The list of cominuscule simple roots in each type is well known. We
    indicate the cominuscule simple roots for each Dynkin diagram (up to
    diagram automorphism) in the left column of Table \ref{TBL:comin}, and
    the corresponding untwisted affine Dynkin diagram in the right column.
    In each case the Dynkin diagram of $W_{S_0}$ is isomorphic to the Dynkin
    diagram of $W_{S_m}$, and this isomorphism identifies $\alpha_m$ with the
    affine root $\alpha_0$.  Consequently $\mcX_J(w_0)$ and $\mcX_J(w_m)$ are
    isomorphic.
\end{proof}

\subsection{Results for cominuscule varieties}
Consider the Schubert variety $Y := \mcX_J(w_0 w_m)$ in $\mcX_J$. The Kac-Moody group
$\mcG$ acts on $\mcX_J$ by left multiplication, and since $G_0$ is the Levi
subgroup of $\mcP_{S_0} \subset \mcG$, we can regard $Y$ as a $G_0$-variety.

In fact $Y$ can naturally be considered as a  $G_0$-homogeneous fibre bundle over $X$.  More precisely:
\begin{theorem}\label{T:main_a}\  
    The affine Schubert variety $Y = \mcX_J(w_0 w_m)$ is stable under the left
    action of $G_0 \subset \mcG$, and the natural projection $Y \to
    \mcX_{S_0}(w_0) \cong X$ is a $G_0$-homogeneous fibre bundle map with fibre
    $\mcX_J(w_m)$. In particular $Y$ is smooth. 
\end{theorem}
Our main result is that if $X$ is cominuscule then $Y$ is a natural
compactification of the cotangent bundle $T^* X$: 
\begin{theorem}\label{T:main_b}
    If $X$ is cominuscule, then the fibre $\mcX_J(w_m)$ is isomorphic to $X$,
    and there is a $G_0$-equivariant map $\tilde{\mu}: T^*X \to Y$ of fibre
    bundles over $X$, under which $T^*X$ is isomorphic to a dense open subset
    of $Y$.
\end{theorem}
We prove Theorem \ref{T:main_a} in Section \ref{S:fibre} and Theorem
\ref{T:main_b} in Section \ref{S:proof}.  In order to prove Theorem
\ref{T:main_b} we explicitly construct the $G_0$-equivariant embedding
$\tilde{\mu}: T^*X \to Y$, which maps the base $X$ isomorphically onto the
Schubert variety $\mcX_J(w_0)$, and maps the fibre over the identity to a dense
open subset of the Schubert variety $\mcX_J(w_m)$.

When $X$ is minuscule rather than cominuscule, it is natural to replace $\mcG$
with a twisted affine Kac-Moody group. Theorem \ref{T:main_a} still holds in
this case, but as we show in Section \ref{S:minuscule}, Theorem
\ref{T:main_b} does not hold.  In this case the variety $Y$ is not the compactification of
the cotangent bundle $T^*X$, but of a different bundle over $X$. 

\subsection{Acknowledgements}
The second author thanks Calin Iuliu Lazaroiu and K. N. Raghavan for useful discussions. 
The Dynkin diagrams in tables \ref{TBL:comin} and \ref{TBL:min} are based on the excellent
TikZ templates due to Oscar Castillo-Felisola.
The first author was supported by NSA grant H98230-11-1-0197.
The second author was partially supported by a fellowship from the Infosys Foundation.

\section{The fibre bundle structure on $Y$}\label{S:fibre}

Given $I \subset K \subset S$, we can write any $w \in W^I$ uniquely as $w =
vu$, where $v \in W^K$ and $u \in W_K^I$. In this case the projection $\mcG /
\mcP_I \arr \mcG / \mcP_K$ induces a projection $\mcX_I(w) \arr \mcX_K(v)$, and the
generic fibre of this projection is $\mcX_I(u)$.  We say $w=vu$ is a
\emph{parabolic decomposition} with respect to $K$.

For any $v \in W$, we define $\supp(v) := \{s \in S \mid s \leq v\}$ to be the
set of simple reflections contained in a reduced expression for $v$.  For any
$u \in W$, let $D^I(u) := \{s \in S \mid  su \leq_I u\}$, where $\leq_I$ is the
Bruhat order on $W / W_I$.  We have the following proposition from
\cite[Theorem 2.3 and Proposition 3.2]{RS14}: 
\begin{prop}\label{P:billey_postnikov}
    The projection $\mcX_I(w) \arr \mcX_K(v)$ is a fibre bundle with fibre
    $\mcX_I(u)$ if and only if $\supp(v) \cap K \subset D^I(u)$.
\end{prop}
When the condition $\supp(v) \cap K \subset D^I(u)$ is satisfied, we say that $w
= vu$ is a \emph{Billey-Postnikov decomposition} with respect to $K$.

Recall that for any $s \in S$, we have $sw \leq_I w$ if and only if $\mcX_I(w)$ is
stable under left multiplication by the rank $1$ parahoric subgroup
$\mcP_{\{s\}}$.  It follows that if $L =D^I(w)$, then $\mcX_I(w)$ is stable under
the action of the parahoric subgroup $\mcP_L$ (\cite{BL00}, see also
\cite[Lemma 3.9]{RS14}).

\begin{lemma}\label{L:fibre}
    Let $y = w_0 w_m$, so $Y = \mcX_J(y)$. 
    \begin{enumerate}[(a)]
        \item $y = w_0 w_m$ is a Billey-Postnikov decomposition with respect to $S_m$.
        \item $D^J(y) = S_0$. 
    \end{enumerate}
\end{lemma}

\begin{proof}
    Since $w_i$ is maximal in $W_{S_i}^J$, we know that $D^J(w_i) = S_i$, for
    $i \in \{0,m\}$. It is clear that $w_0 w_m$ is a parabolic decomposition
    with respect to $S_m$, and $\supp(w_0) \cap S_m = S_0 \cap S_m = J \subset
    D^J(w_m)$, proving part (a).

    For part (b), if $z \in W_{S_0}$ then $z w_0 \leq_J w_0$, and hence $z w_0
    = v_0 z'$, where $v_0 \in W^J_{S_0}$ and $z' \in W_J$. Similarly $z' w_m =
    v_m z''$, where $v_m \in W^J_{S_m}$ and $z'' \in W_J$. So $z y = v_0 v_m
    z'' \leq_J y$, and hence $S_0 \subset D^J(y)$. But $D^J(y)$ must be a
    proper subset of $S$, so $D^J(y) = S_0$. 
\end{proof}
Given $K \subseteq S$, the Levi subgroup $G_K$ of $\mcP_K$ is a Kac-Moody group
with Weyl group $W_K$. Since $\mcG$ is affine, if $K$ is a strict subset of $S$
then $G_K$ is finite-dimensional, and similarly $W_K$ is finite. In order to
prove Theorem 1.2 we need the following standard lemma:
\begin{lemma}\label{L:isomorphic_copies_of_X}
    If $K,I \subsetneq S$ and $w \in W_K^{K \cap I}$, then $P_{K,I} := G_K \cap
    \mcP_I$ is the parabolic subgroup of $G_K$ corresponding to the subgroup
    $W_{K \cap I} \subseteq W_K$, and $\mcX_I(w)$ is isomorphic to a Schubert
    variety in the flag variety $G_K / P_{K,I}$. 
    In particular, if $w$ is the maximal element of $W_K^{K \cap I}$ then
    $\mcX_I(w)$ is isomorphic to $G_K / P_{K,I}$. 
\end{lemma}

Now the proof of Theorem \ref{T:main_a} follows immediately from 
Lemma \ref{L:fibre}:
\begin{proof}[Proof of Theorem \ref{T:main_a}]
    By part (b) of Lemma \ref{L:fibre}, the variety $Y=\mcX_J(w_0w_m)$ is
    stable under the left action of $G_0$.  The base $\mcX_{S_m}(w_0)$ is
    clearly $G_0$-stable as well, and the natural projection $Y \to
    \mcX_{S_m}(w_0)$ is $G_0$-equivariant.  By part (a) of Lemma \ref{L:fibre}
    and Proposition \ref{P:billey_postnikov}, the projection $Y \to
    \mcX_{S_m}(w_0)$ is a $G_0$-homogeneous fibre bundle with fibre
    $\mcX_J(w_m)$. 

    Now the Levi subgroup $G_{S_0}$ of $\mcP_{S_0}$ is simply $G_0$.
    Since $S_m \cap S_0 = J$ and $w_0$ is the maximal element of
    $W_{S_0}^J$, we can alternately set $I = S_m$ and $J$ in Lemma
    \ref{L:isomorphic_copies_of_X} to get $\mcX_{S_m}(w_0) \iso \mcX_{J}(w_0)
    \iso X = G_0 / P_J$. Similarly $\mcX_J(w_m)$ is isomorphic to the flag
    variety $G_{S_m} / P_{S_m, J}$. Since $Y$ is a fibre bundle with smooth
    fibre and base, it follows that $Y$ is smooth.
\end{proof}

\section{The cotangent bundle}\label{S:proof}
If $X$ is cominuscule then $\mcX_J(w_m)$ is isomorphic to $X$ by Lemma
\ref{L:comin}. To prove Theorem \ref{T:main_b}, we explicitly construct the map
$T^* X \incl Y$.  Let $\mcB$ be the Borel subgroup of the Kac-Moody group
$\mcG$ (in the literature $\mcB$ is also known as the Iwahori subgroup of
$\mcG$). For convenience, we write $G_i$ for the Levi subgroup $G_{S_i}$ of the
parahoric subgroup $\mcP_{S_i} \subset \mcG$, where $i \in \{0,m\}$ (in
particular $G_0$ is the same as before).  We let $B_i := G_i \cap \mcB$ be the
induced Borel of $G_i$, and $P_i := B_i W_J B_i = G_i \cap \mcP_{J}$.  Finally,
let $U_i \subset P_i$ be the unipotent radical of $P_i$. As in the previous
section, $P_0 = P_J$, $X = G_0 / P_0$, and moreover $\mcX_J(w_i) \iso G_i /
P_i$ for $i \in \{0,m\}$.

We will also need to use the underlying Lie algebras.  We assume the standard
construction of $\mfg$, in which 
\begin{equation*}
    \mfg \iso \mfg_0 \otimes_\C \C[z,z^{-1}] \oplus \C c \oplus \C d
\end{equation*}
as a vector space (see \cite[\S 18.1]{Ca05} or \cite[\S 7.2]{Ka90}).  Let $\mfh$ be the Cartan
subalgebra of the Kac-Moody algebra $\mfg$.  Let $\mfg_i \subset \mfg$ be the
(finite-dimensional) Lie algebra of $G_i$, where $i\in \{0,m\}$.  Let $\mfu_i
\subset \mfg_i$ be the (nilpotent) Lie algebra of $U_i$.  Finally, let
$\mfu_m^-$ be the opposite nilpotent radical to $\mfu_m$ inside $\mfg_m$.  We
consider the linear map
\begin{equation*}
    \phi : \mfu_0 \arr \mfg \text{ defined by } x \mapsto x \otimes z^{-1}.
\end{equation*}
In order to prove Theorem  \ref{T:main_b} we will need the following lemma.

\begin{lemma}\label{L:phi}
    The map $\phi: \mfu_0 \to \mfu^-_m$ is a $P_0$-equivariant isomorphism of
    vector spaces.
\end{lemma}
\begin{proof}
    Let $R$ denote the set of roots of $\mfg$, with simple roots $\Delta :=
    \{\alpha_0, \ldots, \alpha_n\}$.  The simple roots of $\mfg_0$ and $\mfg_m$
    are the subsets of $\Delta$ obtained by omitting $\alpha_0$ and $\alpha_m$
    respectively.  For any subalgebra $\mathfrak{a} \subset \mfg$, we let
    $R(\mathfrak{a})$ denote the set of $\mfh$-weights of $\mathfrak{a}$, and
    let $R^+(\mathfrak{a})$ and $R^-({\mathfrak{a}})$ denote the subsets of
    positive and negative roots respectively.  Let $\theta$ be the highest root
    of $\mfg_0$, and let $\delta = \alpha_0 + \theta$ be the basic imaginary
    root of $\mfg$ (\cite[\S 17.1]{Ca05} or \cite[\S 5.6]{Ka90}).

    We can describe the set of roots of $\mfg$ by
    \begin{equation*}
        R(\mfg) = \{\alpha + k\delta \mid \alpha \in R(\mfg_0), k \in \Z\} 
            \cup \{k\delta \mid k \in \Z_{\neq0}\}.
    \end{equation*}
    The set of positive roots of $\mfg$ is given by
    \begin{equation*}
        R^+(\mfg) = R^+(\mfg_0) \cup \{\alpha + k\delta \in R \mid \alpha \in
            R(\mfg_0), k \in \Z_{>0}\} \cup \{k\delta \mid k \in \Z_{>0}\}.
    \end{equation*}
    Note that $R(\mfu_0) \subset R^+(\mfg_0)$ and $R(\mfu_m) \subset
    R^+(\mfg_m)$.  Using the simple roots of $\mfg_0$ and $\mfg_m$, the roots
    of $\mfu_0$ and $\mfu_m$ can then be written
    \begin{equation*}
        R(\mfu_0) = \left\{\sum^n_{i=1} a_i \alpha_i \in R^+(\mfg) \mid a_m = 1\right\}, \text{ and}
    \end{equation*}
    \begin{equation*}
        R(\mfu_m) = \left\{a_0\alpha_0 + 
            \sum_{i \in [1,n] \setminus \{m\}} a_i \alpha_i \in R^+(\mfg) \mid a_0 = 1\right\},
    \end{equation*}
    where the requirement that $a_m=1$ (resp. $a_0 = 1$) follows from the fact
    that $\alpha_m$ is cominuscule in $\mfg_0$ (resp. $\alpha_0$ is cominuscule
    in $\mfg_m$).

    Every root of $\mfg_0$ can be written uniquely as $\theta - \sum^n_{i=1}
    a_i \alpha_i$ where $a_i \geq 0$ for all $1 \leq i \leq n$.  Since
    $\alpha_m$ is cominuscule, the coefficient of  $\alpha_m$ in $\theta$ is
    equal to $1$.  Using the previous description of $R(\mfu_0)$, it follows
    that $\alpha \in R(\mfg_0)$ is an element of $R(\mfu_0)$ if and only if
    \begin{equation*}
        \alpha = \theta - \sum_{i \in [1,n] \setminus \{m\}} a_i \alpha_i
    \end{equation*}
    for some coefficients $a_i \geq 0$ (in particular, note that any $\alpha$
    of this form cannot belong to $R^-(\mfg_0)$, since it will have positive
    $\alpha_m$-coefficient). 

    Note that for any $\alpha \in R(\mfu_0)$, the homomorphism $\phi$  maps
    $\mfg_\alpha$ isomorphically onto $\mfg_{\alpha -\delta}$.  Thus the
    $\mfh$-weights of $\phi(\mfu_0)$ are precisely
    \begin{multline*}
        \left\{\alpha - \delta \in R \mid  \alpha = \theta - \sum_{i \in [1,n]
            \setminus \{m\}} a_i \alpha_i, \text{ where } a_i \geq 0 
            \text{ for } i \in [1,n] \setminus \{m\}\right\} \\
        = \left\{ -(-\theta + \delta) - \sum_{i \in [1,n] \setminus \{m\}} a_i  \alpha_i 
            \in R^- \mid a_i \geq 0 \text{ for all } i \in [1,n] \setminus \{m\} \right\}.
    \end{multline*}
    This latter set is exactly the negative of the $\mfh$-weights of $\mfu_m$,
    since $(-\theta + \delta) = \alpha_0$, and  since $\alpha_0$ is cominuscule
    in $\mfg_m$ as in Lemma \ref{L:comin}.  We conclude that $\phi(\mfu_0) =
    \mfu_m^-$.  Since $\phi$ is a clearly bijective, it is a vector space
    isomorphism.

    Consider the left adjoint action of $\mfp_0 := \Lie(P_0)$ on $\mfg$.  Under
    this action, each element of the weight space $\mfg_\beta \subset \mfp_0$
    maps $\mfg_\alpha$ into $\mfg_{\alpha+\beta}$ whenever $\alpha + \beta \in
    R(\mfg)$, and annihilates  $\mfg_\alpha$ otherwise.  Recall that $R(\mfp_0)
    = R^+(\mfg_0) \cup \{\sum^n_{i=1} a_i\alpha_i \in R^-(\mfg_0) \mid a_m =
    0\}$, and observe that both $\mfu_0$ and $\mfu^-_m$ are stable under the
    left adjoint action of $\mfp_0$, and moreover that $\phi$ is
    $\mfp_0$-equivariant.  It follows that $\phi$ is $P_0$-equivariant.
\end{proof}

Using the map $\phi:\mfu_0 \to \mfu^-_m$, we construct a map
\begin{equation*}
    \Phi : \mfu_0 \arr \mcX_J = \mcG / \mcP_J : x \mapsto \left[ \exp(\phi(x)) \cdot \mcP_J \right].
\end{equation*}

\begin{lemma}\label{L:fibrebundle}
    $\Phi$ is a $P_0$-equivariant algebraic isomorphism from $\mfu_0$ to an
    open dense subset of $\mcX_J(w_m)$.
\end{lemma}

\begin{proof}
    The exponential map $\mfu_m^- \arr \exp(\mfu_m^-) =: U^-_m$ is an algebraic
    isomorphism, and $U^-_m \iso U^-_m \cdot [e\mcP_J]$ is an open dense subset
    of $\mcX_J(w_m) = G_m / P_m$, where $e \in \mcG$ is the identity.  Since
    $\phi$ is a $P_0$-equivariant bijection and $P_0 \subset \mcP_J$, the result follows.
\end{proof}

We can now finish the proof of the main theorem.
\begin{proof}[Proof of Theorem \ref{T:main_b}]
    As in Section \ref{S:fibre}, we let $Y = \mcX_J(y)$, where $y = w_0 w_m$. The
    cotangent bundle of $X$ is
    \begin{equation*}
        T^* X = G_0 \times_{P_0} \mfu_0,
    \end{equation*}
    the quotient of $G_0 \times \mfu_0$ by the $P_0$-action $p \cdot (g,x) = (g p,
    p^{-1} x)$. We can define a map
    \begin{equation*}
        \mu : G_0 \times \mfu_0 \arr \mcX_J(y) : (g,x) \mapsto g \cdot \Phi(x),
    \end{equation*}
    where we use the fact that $\Phi(x) \in \mcX_J(w_m) \subset \mcX_J(y)$, which is
    stable under the left action of $G_0$ by Theorem \ref{T:main_a}. But
    $\mu$ is $P_0$-equivariant, so we get an induced map
    \begin{equation*}
        \widetilde{\mu} : G_0 \times_{P_0} \mfu_0 \arr \mcX_J(y).
    \end{equation*} 
    The cotangent bundle map $T^* X \arr X$ sends $(g,x) \mapsto \left[ g P_0
    \right]$.  Since the projection $\mcG \mapsto \mcG/\mcP_{S_m}$ sends $g
    \cdot \Phi(x) \mapsto \left[ g \mcP_{S_m} \right]$, we conclude that the
    diagram
    \begin{equation*}
        \xymatrix{ G_0 \times_{P_0} \mfu_0 \ar[rr] \ar[rd] & & \mcX_J(y) \ar[ld] \\
                        & G_0/P_0 \iso \mcX_{S_m}(w_0) & }
    \end{equation*}
    commutes, and thus $\widetilde{\mu}$ is a morphism of $G_0$-homogeneous
    fibre bundles. Over $\left[ e P_0 \right]$, this map restricts to $\Phi :
    \mfu_0 \arr \mcX_J(w_m)$, which is injective and has open dense image in the
    fibre $\mcX_J(w_m)$. We conclude that the total map $G_0 \times_{P_0} \mfu_0
    \arr \mcX_J(y)$ is injective and has open image.  
\end{proof}

\section{Minuscule Grassmannians}\label{S:minuscule}
A Grassmannian $X = G_0 / P_{S_{0,m}}$ is \emph{minuscule} if $\alpha^\vee_m$
is cominuscule in the dual root system.  The minuscule and cominuscule
Grassmannians coincide in types $A$, $D$, and $E$, but are disjoint in the
other types. There are just two families of Grassmannians which are minuscule
but not cominuscule: $\SO(2n+1) / P_{S_{0,n}}$, the Grassmannian corresponding
to the root $\alpha_n$ in type $B_n$, and $\Sp(2n) / P_{S_{0,1}}$, the
Grassmannian corresponding to the root $\alpha_1$ in $C_n$. The corresponding
Dynkin diagrams are listed in Table \ref{TBL:min}. As algebraic varieties,
$\Sp(2n)/P_{S_{0,1}}$ is isomorphic to $\mathbb{P}^{2n-1}$ and
$\SO(2n+1)/P_{S_{0,n}}$ is isomorphic to $\SO(2n+2)/P_n \iso \SO(2n+2)/P_{n+1}$,
so each minuscule Grassmannians is isomorphic to a cominuscule Grassmannian.
However, the minuscule Grassmannians are distinct as homogeneous spaces, and
their cotangent bundles are distinct as homogeneous bundles. 

Suppose $\alpha_m$ is minuscule but not cominuscule, and let $\mfg$ and $\mcG$
be the affine \emph{twisted} Kac-Moody algebra and group associated to $\mfg_0$
(see \cite[\S 18.4]{Ca05} or \cite{Ka90}, and \cite[\S 6]{Ku02}).  The proof of
Theorem \ref{T:main_a} still works in this setting, and consequently the affine
Schubert variety $Y = \mcX_J(w_0 w_1) \subseteq \mcX_J := \mcG / \mcP_J$ is a
fibre bundle over $X$ with fibre $\mcX_J(w_m)$. Furthermore,
following Lemma \ref{L:comin}, 
we have $\mcX_J(w_m) \iso \mcX_J(w_0) \iso X$ (see Table
\ref{TBL:min} for the proof). 

With all these pieces in place, we might expect that $Y$ is a compactification of $T^*
X$ as in the cominuscule case. However, the argument from the
cominuscule setting breaks down at this point. Specifically, the argument from Section
\ref{S:proof} shows that $Y$ is a compactification of the homogeneous vector
bundle $T := G_0 \times_{P_0} \mfu_m^-$ on $X$. 
However, $T$ is not the
cotangent bundle of $X$. 
Indeed, by the following Lemma, $T$ splits as the direct sum of two $G_0$-homogeneous
vector bundles on $X$, whereas $T^*X$ does not.
\begin{lemma}\label{L:distinct_modules}
As $P_0$-modules, $\mfu^-_m$ splits as the direct sum of two submodules,
while $\mfu_0$ does not.
\end{lemma}
\begin{proof}
Let $\delta = \alpha_0 + \theta_0$ be the
basic imaginary root of $\mfg$, where $\theta_0$ is the highest \emph{short}
root of $\mfg_0$ (\cite[\S 17.1]{Ca05} or \cite[\S 8.3]{Ka90}).  For any subalgebra $\mathfrak{a} \subset \mfg$, 
let $R_{s}(\mathfrak{a})$ (resp.
$R_{l}(\mathfrak{a})$) denote the set of real short (resp. long) $\mfh$-weights of
$\mathfrak{a}$  (see \cite[\S 17.2]{Ca05} or \cite[\S 5.1]{Ka90}). The set of roots of $\mfg$ is given by
\[
    \{\alpha + k\delta: \alpha \in R_{s}(\mfg_0), k \in \Z\} 
        \cup \{\alpha + 2k\delta: \alpha \in R_{l}(\mfg_0), k \in \Z\} 
        \cup \{k\delta: k \in \Z_{\neq0}\}.
\]
Moreover, the $\mfh$-weights of $\mfu_0$ and $\mfu^-_m$ are given by
\[
R(\mfu_0) =  \left\{\sum_{i \in [1,n]} a_i \alpha_i \in R^+(\mfg): a_m \in \{1,2\}\right\}, \text{and }
\]
\[
R(\mfu^-_m) =  \left\{a_0\alpha_0 + \sum_{i \in [1,n] \setminus \{m\}} a_i \alpha_i \in R^-(\mfg): a_0 \in \{-1,-2\}\right\}.
\]
Write $\mfu^-_m = \mfu^-_{m,s} \oplus \mfu^-_{m,l}$, where $\mfu^-_{m,s} :=
\oplus_{\alpha \in R_s(\mfu^-_m)}\mfg_\alpha$ and $\mfu^-_{m,l} :=
\oplus_{\alpha \in R_l(\mfu^-_m)}\mfg_\alpha$.  The short (resp. long) $\mfh$-weights of
$\mfu^-_{m}$ are precisely those with $\alpha_0$ coefficient $a_0 = -1$ (resp.
$a_0 = -2$)  in the simple root basis.  It follows that the left adjoint action of
$P_0$ preserves the long and short roots of $u^-_m$, and hence $T = G_0
\times_{P_0} u_{m,s}^- \oplus G_0 \times_{P_0} u_{m,l}^-$ is a direct sum of
two homogeneous vector bundles. 
On the other hand $\mfu_0$ does not split as a $P_0$-module, since the Lie 
algebra $\mfp_0$ of $P_0$ can take short roots of $\mfu_0$ (which have $\alpha_m$
coefficient $1$)  to long roots (which have $\alpha_m$ coefficient
$2$).
\end{proof}

Let $\mfh_0$ and $H_0$ denote the Cartan subalgebra and subgroup of $\mfg_0$ and $G_0$ respectively.
An $H_0$-module $M$ is \emph{attractive} if there is some $\omega$ in $\mfh_0$ such that 
$\alpha(\omega) > 0$ for all $\mfh_0$-weights $\alpha$ of $M$.
The fact that $Y$ cannot be the 
compactification of $T^*X$ follows from the following more general result.
\begin{lemma}\label{L:attractive}
Given $P_0$-modules $U$ and $V$, suppose there exists an element $\omega \in \mfh_0$ with the property
that $\alpha(\omega) > 0$ for any $\mfh_0$-weight $\alpha$ of $U$ or $V$.
Furthermore, if both $G_0 \times_{P_0} U$ and $G_0 \times_{P_0} V$
embed as open dense homogeneous $G_0$-bundles in a homogeneous $G_0$-fibre bundle $Y$, 
then $U$ and $V$ are isomorphic as $P_0$-modules.
\end{lemma}
\begin{proof}
 We can think of $U$ and $V$ as open dense subsets of the fibre over the identity in $Y$. As such, the intersection of 
 $U$ and $V$ is non-empty. 
 Let $y$ be a point of the intersection.  Since $\alpha(\omega) > 0$ for all $\mfh_0$-weights $\alpha$ of $U$ or $V$,
 the limit 
 $$\lim_{n \to \infty} \exp(-n \omega).y$$ 
 exists and is equal to both $0_U$ and $0_V$, the zero elements of $U$ and $V$,
 which in particular must be equal. 
 The sets $U$ and $V$ are both open, and $0:=0_U=0_V$ is a $P_0$-fixed point in both $U$ and $V$, so
 $U \iso T_0 U = T_0 V \iso V$ as $P_0$-modules. 
\end{proof}

\begin{corollary}
There is no open embedding of $T^*X$ into $Y$
as $G_0$-homogeneous fibre bundles over $X$.
\end{corollary}
\begin{proof}
 Suppose on the contrary that such an embedding exists.
 Note that $\alpha(\omega_m) > 0$ for all $\alpha \in R(\mfu_0)$, where
 $\omega_m \in \mfh_0$ is the fundamental weight dual to the simple root $\alpha_m$.
 Moreover, the roots of $\mfu^-_m$ are all of the form $\alpha - \delta$ or
 $\alpha - 2\delta$ with $\alpha \in R(\mfu_0)$.  Since $\delta(\omega_m) = 0$,
 it follows that $\beta(\omega_m) > 0$ for all $\beta \in R(\mfu^-_m)$
 (indeed, $\mfu_0$ and $\mfu^-_m$ have the same $\mfh_0$-weights since $\delta(\omega)=0$ for
 any $\omega \in \mfh_0$).
 By Lemma \ref{L:attractive}, it follows that $\mfu_0$ and $\mfu^-_m$
 are isomorphic as $P_0$-modules,
 contradicting Lemma \ref{L:distinct_modules}.
\end{proof}

\bibliographystyle{alpha}

\begin{thebibliography}{XX}

\bibitem{BL00}
Sara Billey and V. Lakshmibai.
\newblock Singular Loci of Schubert Varieties.
\newblock Birkh\"auser, 2000.

\bibitem{Ca05}
Roger Carter.
\newblock Lie Algebras of Finite and Affine Type.
\newblock Cambridge University Press, 2005.

\bibitem{Ka90}
Victor Kac.
\newblock Infinite Dimensional Lie Algebras, Third Edition.
\newblock Cambridge University Press, 1990.

\bibitem{Ku02}
Shrawan Kumar.
\newblock Kac-Moody Groups, their Flag Varieties and Representation Theory.
\newblock Birkh\"auser, 2002.

\bibitem{La14}
V. Lakshmibai.
\newblock \emph{Cotangent bundle to the Grassmann variety}.
\newblock Preprint, arXiv:1505.00038, 2015. 

\bibitem{LS78} V. Lakshmibai and C.S. Seshadri. 
\newblock \emph{Geometry of $G/P-II$}. 
\newblock Proc. Ind. Acad. Sci., 87A (1978), 1-54.

\bibitem{Lu90} G. Lusztig. 
\newblock \emph{Canonical bases arising from quantized enveloping algebras}. 
\newblock J. Amer. Math. Soc. 3, 1990, 447-498.

\bibitem{RS14}
Edward Richmond and William Slofstra.
\newblock \emph{Billey-Postnikov decompositions and the fibre bundle structure of Schubert varieties}.
\newblock Preprint, arXiv:1408.0084, 2014. 

\bibitem{St82} E. Strickland. 
\newblock\emph{On the conormal bundle of the determinantal variety}. 
\newblock J. Algebra, vol 75(1982), 523-537.
\end{thebibliography}

\begin{table}

  \begin{tikzpicture}[scale=.4]
    \draw (-1,0) node[anchor=east]  {$A_{n}$};
    \foreach \x in {0,2,6,8}
    \draw[thick,fill=white!70] (\x cm,0) circle (.3cm);
    \draw[thick,fill=black!70] (4 cm,0) circle (.3cm);
    \draw[dotted, thick] (0.3 cm,0) -- +(1.4 cm,0);
    \foreach \y in {2.3,4.3}
    \draw[thick] (\y cm,0) -- +(1.4 cm,0);
    \draw[dotted, thick] (6.3 cm,0) -- +(1.4 cm,0);
    \draw (0,.8) node {$\scriptscriptstyle{1}$};
    \draw (2,.8) node {$\scriptscriptstyle{m-1}$};
    \draw (4,.8) node {$\scriptscriptstyle{m}$};
    \draw (6,.8) node {$\scriptscriptstyle{m+1}$};
    \draw (8,.8) node {$\scriptscriptstyle{n}$};

    \draw (15,0) node[anchor=east]  {$\tilde{A}_{n}$};
    \foreach \x in {16,18,20,22,24}
    \draw[thick,fill=white!70] (\x cm,0) circle (.3cm);
    \draw[thick, fill=black!70] (20 cm,-2 cm) circle (.3cm);
    \draw[dotted,thick] (16.3 cm,0) -- +(1.4 cm,0);   
    \foreach \y in {18.3,20.3}
    \draw[thick] (\y cm,0) -- +(1.4 cm,0);
    \draw[dotted,thick] (22.3 cm,0) -- +(1.4 cm,0); 
    \draw[thick] (16.2 cm,-0.2 cm) -- +(3.5 cm,-1.65 cm); 
    \draw[thick] (20.3 cm,-1.9 cm) -- +(3.5 cm,1.65 cm); 
    \draw (16,.8) node {$\scriptscriptstyle{1}$};
    \draw (18,.8) node {$\scriptscriptstyle{m-1}$};
    \draw (20,.8) node {$\scriptscriptstyle{m}$};
    \draw (22,.8) node {$\scriptscriptstyle{m+1}$};
    \draw (24,.8) node {$\scriptscriptstyle{n}$};
    \draw (20,-1.2) node {$\scriptscriptstyle{0}$};
    
  \end{tikzpicture}

  \vspace{3mm}

  \begin{tikzpicture}[scale=.4]
    \draw (-1,0) node[anchor=east]  {$B_{n}$};
    \foreach \x in {2,4,6,8}
    \draw[thick,fill=white!70] (\x cm,0) circle (.3cm);
    \draw[thick,fill=black!70] (0 cm,0) circle (.3cm);
    \draw[thick] (0.3 cm,0) -- +(1.4 cm,0);
    \draw[dotted,thick] (2.3 cm,0) -- +(1.4 cm,0);
    \draw[thick] (4.3 cm,0) -- +(1.4 cm,0);
    \draw[thick] (6.3 cm, .1 cm) -- +(1.4 cm,0);
    \draw[thick] (6.3 cm, -.1 cm) -- +(1.4 cm,0);
    \draw[thick] (6.9 cm, .3 cm) -- +(.3 cm, -.3 cm);
    \draw[thick] (6.9 cm, -.3 cm) -- +(.3 cm, .3 cm);
    \draw (0,.8) node {$\scriptscriptstyle{1}$};
    \draw (2,.8) node {$\scriptscriptstyle{2}$};
    \draw (4,.8) node {$\scriptscriptstyle{n-2}$};
    \draw (6,.8) node {$\scriptscriptstyle{n-1}$};
    \draw (8,.8) node {$\scriptscriptstyle{n}$};
    
    \draw (15,0) node[anchor=east]  {$\tilde{B}_{n}$};
    \foreach \x in {18,20,22,24}
    \draw[thick,fill=white!70] (\x cm,0) circle (.3cm);
    \draw[xshift=15 cm,thick,fill=black!70] (30: 17 mm) circle (.3cm);
    \draw[xshift=15 cm,thick,fill=black!70] (-30: 17 mm) circle (.3cm);
    \draw[dotted,thick] (18.3 cm,0) -- +(1.4 cm,0);
    \foreach \y in {20.3}
    \draw[thick] (\y cm,0) -- +(1.4 cm,0);
    \draw[thick] (22.3 cm, .1 cm) -- +(1.4 cm,0);
    \draw[thick] (22.3 cm, -.1 cm) -- +(1.4 cm,0);
    \draw[thick] (22.9 cm, .3 cm) -- +(.3 cm, -.3 cm);
    \draw[thick] (22.9 cm, -.3 cm) -- +(.3 cm, .3 cm);
    \draw[xshift=17.5 cm,thick] (30: 3 mm) -- (136: 10.5 mm);
    \draw[xshift=17.5 cm,thick] (-30: 3 mm) -- (-136: 10.5 mm);
    \draw (16.5,1.5) node {$\scriptscriptstyle{1}$};
    \draw (16.5,-.2) node {$\scriptscriptstyle{0}$};
    \draw (18,.8) node {$\scriptscriptstyle{2}$};
    \draw (20,.8) node {$\scriptscriptstyle{n-2}$};
    \draw (22,.8) node {$\scriptscriptstyle{n-1}$};
    \draw (24,.8) node {$\scriptscriptstyle{n}$};
  \end{tikzpicture}
  
    \vspace{3mm}

  \begin{tikzpicture}[scale=.4]
    \draw (-1,0) node[anchor=east]  {$C_{n}$};
    \foreach \x in {0,2,4,6}
    \draw[thick,fill=white!70] (\x cm,0) circle (.3cm);
    \draw[thick,fill=black!70] (8 cm,0) circle (.3cm);
    \draw[thick] (0.3 cm,0) -- +(1.4 cm,0);
    \draw[dotted,thick] (2.3 cm,0) -- +(1.4 cm,0);
    \draw[thick] (4.3 cm,0) -- +(1.4 cm,0);
    \draw[thick] (6.3 cm, .1 cm) -- +(1.4 cm,0);
    \draw[thick] (6.3 cm, -.1 cm) -- +(1.4 cm,0);
    \draw[thick] (6.9 cm, 0 cm) -- +(.3 cm, .3 cm);
    \draw[thick] (6.9 cm, 0 cm) -- +(.3 cm, -.3 cm);
    \draw (0,.8) node {$\scriptscriptstyle{1}$};
    \draw (2,.8) node {$\scriptscriptstyle{2}$};
    \draw (4,.8) node {$\scriptscriptstyle{n-2}$};
    \draw (6,.8) node {$\scriptscriptstyle{n-1}$};
    \draw (8,.8) node {$\scriptscriptstyle{n}$};
    
    \draw (15,0) node[anchor=east]  {$\tilde{C}_{n}$};
    \foreach \x in {16,18,20,22,24,26}
    \draw[thick,fill=white!70] (\x cm,0) circle (.3cm);
    \foreach \x in {16,26}
    \draw[thick,fill=black!70] (\x cm,0) circle (.3cm);
    \draw[dotted,thick] (20.3 cm,0) -- +(1.4 cm,0);
    \foreach \y in {18.3,22.3}
    \draw[thick] (\y cm,0) -- +(1.4 cm,0);
    \draw[thick] (24.3 cm, .1 cm) -- +(1.4 cm,0);
    \draw[thick] (24.3 cm, -.1 cm) -- +(1.4 cm,0);
    \draw[thick] (24.9 cm, 0 cm) -- +(.3 cm, .3 cm);
    \draw[thick] (24.9 cm, 0 cm) -- +(.3 cm, -.3 cm);
    \draw[thick] (16.3 cm, .1 cm) -- +(1.4 cm,0);
    \draw[thick] (16.3 cm, -.1 cm) -- +(1.4 cm,0);
    \draw[thick] (17.2 cm, 0 cm) -- +(-.3 cm, .3 cm);
    \draw[thick] (17.2 cm, 0 cm) -- +(-.3 cm, -.3 cm);
    \draw (16,.8) node {$\scriptscriptstyle{0}$};
    \draw (18,.8) node {$\scriptscriptstyle{1}$};
    \draw (20,.8) node {$\scriptscriptstyle{2}$};
    \draw (22,.8) node {$\scriptscriptstyle{n-2}$};
    \draw (24,.8) node {$\scriptscriptstyle{n-1}$};
    \draw (26,.8) node {$\scriptscriptstyle{n}$};
  \end{tikzpicture}

  \vspace{3mm}

  \begin{tikzpicture}[scale=.4]
    \draw (-1,0) node[anchor=east]  {$D_{n}$};
    \foreach \x in {2,4,6,8}
    \draw[thick,fill=white!70] (\x cm,0) circle (.3cm);
    \draw[thick,fill=black!70] (0 cm,0) circle (.3cm);
    \draw[xshift=8 cm,thick,fill=white!70] (30: 17 mm) circle (.3cm);
    \draw[xshift=8 cm,thick,fill=white!70] (-30: 17 mm) circle (.3cm);
    \draw[dotted,thick] (4.3 cm,0) -- +(1.4 cm,0);
    \foreach \y in {0.3, 2.3,6.3}
    \draw[thick] (\y cm,0) -- +(1.4 cm,0);
    \draw[xshift=8 cm,thick] (30: 3 mm) -- (30: 14 mm);
    \draw[xshift=8 cm,thick] (-30: 3 mm) -- (-30: 14 mm);
    \draw (0,.8) node {$\scriptscriptstyle{1}$};
    \draw (2,.8) node {$\scriptscriptstyle{2}$};
    \draw (4,.8) node {$\scriptscriptstyle{3}$};
    \draw (6,.8) node {$\scriptscriptstyle{n-3}$};
    \draw (8,.8) node {$\scriptscriptstyle{n-2}$};
    \draw (9.45,1.6) node {$\scriptscriptstyle{n-1}$};
    \draw (9.45,-.2) node {$\scriptscriptstyle{n}$};
    
    \draw (15,0) node[anchor=east]  {$\tilde{D}_{n}$};
    \foreach \x in {18,20,22,24}
    \draw[thick,fill=white!70] (\x cm,0) circle (.3cm);
    \draw[xshift=24 cm,thick,fill=white!70] (30: 17 mm) circle (.3cm);
    \draw[xshift=24 cm,thick,fill=white!70] (-30: 17 mm) circle (.3cm);
    \draw[xshift=15 cm,thick,fill=black!70] (30: 17 mm) circle (.3cm);
    \draw[xshift=15 cm,thick,fill=black!70] (-30: 17 mm) circle (.3cm);
    \draw[dotted,thick] (20.3 cm,0) -- +(1.4 cm,0);
    \foreach \y in {18.3, 22.3}
    \draw[thick] (\y cm,0) -- +(1.4 cm,0);
    \draw[xshift=24 cm,thick] (30: 3 mm) -- (30: 14 mm);
    \draw[xshift=24 cm,thick] (-30: 3 mm) -- (-30: 14 mm);
    \draw[xshift=17.5 cm,thick] (30: 3 mm) -- (136: 10.5 mm);
    \draw[xshift=17.5 cm,thick] (-30: 3 mm) -- (-136: 10.5 mm);
    \draw (16.5,1.5) node {$\scriptscriptstyle{1}$};
    \draw (16.5,-.2) node {$\scriptscriptstyle{0}$};
    \draw (18,.8) node {$\scriptscriptstyle{2}$};
    \draw (20,.8) node {$\scriptscriptstyle{3}$};
    \draw (22,.8) node {$\scriptscriptstyle{n-3}$};
    \draw (24,.8) node {$\scriptscriptstyle{n-2}$};
    \draw (25.45,1.6) node {$\scriptscriptstyle{n-1}$};
    \draw (25.45,-.2) node {$\scriptscriptstyle{n}$};
  \end{tikzpicture}

  \vspace{3mm}

\begin{tikzpicture}[scale=.4]
    \draw (-1,0) node[anchor=east]  {$D_{n}$};
    \foreach \x in {0,2,4,6,8}
    \draw[thick,fill=white!70] (\x cm,0) circle (.3cm);
    \draw[xshift=8 cm,thick,fill=white!70] (30: 17 mm) circle (.3cm);
    \draw[xshift=8 cm,thick,fill=black!70] (-30: 17 mm) circle (.3cm);
    \draw[dotted,thick] (4.3 cm,0) -- +(1.4 cm,0);
    \foreach \y in {0.3, 2.3,6.3}
    \draw[thick] (\y cm,0) -- +(1.4 cm,0);
    \draw[xshift=8 cm,thick] (30: 3 mm) -- (30: 14 mm);
    \draw[xshift=8 cm,thick] (-30: 3 mm) -- (-30: 14 mm);
    \draw (0,.8) node {$\scriptscriptstyle{1}$};
    \draw (2,.8) node {$\scriptscriptstyle{2}$};
    \draw (4,.8) node {$\scriptscriptstyle{3}$};
    \draw (6,.8) node {$\scriptscriptstyle{n-3}$};
    \draw (8,.8) node {$\scriptscriptstyle{n-2}$};
    \draw (9.45,1.6) node {$\scriptscriptstyle{n-1}$};
    \draw (9.45,-.2) node {$\scriptscriptstyle{n}$};
    
    \draw (15,0) node[anchor=east]  {$\tilde{D}_{n}$};
    \foreach \x in {18,20,22,24}
    \draw[thick,fill=white!70] (\x cm,0) circle (.3cm);
    \draw[xshift=24 cm,thick,fill=white!70] (30: 17 mm) circle (.3cm);
    \draw[xshift=24 cm,thick,fill=black!70] (-30: 17 mm) circle (.3cm);
    \draw[xshift=15 cm,thick,fill=white!70] (30: 17 mm) circle (.3cm);
    \draw[xshift=15 cm,thick,fill=black!70] (-30: 17 mm) circle (.3cm);
    \draw[dotted,thick] (20.3 cm,0) -- +(1.4 cm,0);
    \foreach \y in {18.3, 22.3}
    \draw[thick] (\y cm,0) -- +(1.4 cm,0);
    \draw[xshift=24 cm,thick] (30: 3 mm) -- (30: 14 mm);
    \draw[xshift=24 cm,thick] (-30: 3 mm) -- (-30: 14 mm);
    \draw[xshift=17.5 cm,thick] (30: 3 mm) -- (136: 10.5 mm);
    \draw[xshift=17.5 cm,thick] (-30: 3 mm) -- (-136: 10.5 mm);
    \draw (16.5,1.5) node {$\scriptscriptstyle{1}$};
    \draw (16.5,-.2) node {$\scriptscriptstyle{0}$};
    \draw (18,.8) node {$\scriptscriptstyle{2}$};
    \draw (20,.8) node {$\scriptscriptstyle{3}$};
    \draw (22,.8) node {$\scriptscriptstyle{n-3}$};
    \draw (24,.8) node {$\scriptscriptstyle{n-2}$};
    \draw (25.45,1.6) node {$\scriptscriptstyle{n-1}$};
    \draw (25.45,-.2) node {$\scriptscriptstyle{n}$};
  \end{tikzpicture}

    \vspace{3mm}

  \begin{tikzpicture}[scale=.4]
    \draw (-1,0) node[anchor=east]  {$E_{6}$};
    \foreach \x in {2,4,6,8}
    \draw[thick,fill=white!70] (\x cm,0) circle (.3cm);
    \draw[thick,fill=white!70] (4 cm, 2 cm) circle (.3cm);
    \draw[thick,fill=black!70] (0 cm,0) circle (.3cm);
    \foreach \y in {0.3, 2.3, 4.3, 6.3}
    \draw[thick] (\y cm,0) -- +(1.4 cm,0);
    \draw[thick] (4 cm,.3 cm) -- +(0,1.4 cm);
    \draw (0,.8) node {$\scriptscriptstyle{1}$};
    \draw (2,.8) node {$\scriptscriptstyle{2}$};
    \draw (4.5,.8) node {$\scriptscriptstyle{3}$};
    \draw (6,.8) node {$\scriptscriptstyle{4}$};
    \draw (8,.8) node {$\scriptscriptstyle{5}$};
    \draw (4.5,2.8) node {$\scriptscriptstyle{6}$};
    
    \draw (15,0) node[anchor=east]  {$\tilde{E}_{6}$};
    \foreach \x in {18,20,22,24}
    \draw[thick,fill=white!70] (\x cm,0) circle (.3cm);
    \draw[thick,fill=white!70] (20 cm, 2 cm) circle (.3cm);
    \draw[thick,fill=black!70] (20 cm, 4 cm) circle (.3cm);
    \draw[thick,fill=black!70] (16 cm,0) circle (.3cm);
    \foreach \y in {16.3, 18.3, 20.3, 22.3}
    \draw[thick] (\y cm,0) -- +(1.4 cm,0);
    \draw[thick] (20 cm,.3 cm) -- +(0,1.4 cm);
    \draw[thick] (20 cm,2.3 cm) -- +(0,1.4 cm);
    \draw (16,.8) node {$\scriptscriptstyle{1}$};
    \draw (18,.8) node {$\scriptscriptstyle{2}$};
    \draw (20.5,.8) node {$\scriptscriptstyle{3}$};
    \draw (22,.8) node {$\scriptscriptstyle{4}$};
    \draw (24,.8) node {$\scriptscriptstyle{5}$};
    \draw (20.5,2.8) node {$\scriptscriptstyle{6}$};
    \draw (20.5,4.8) node {$\scriptscriptstyle{0}$};
  \end{tikzpicture}

\vspace{3mm}

  \begin{tikzpicture}[scale=.4]
    \draw (-1,0) node[anchor=east]  {$E_{7}$};
    \foreach \x in {0,2,4,6,8}
    \draw[thick,fill=white!70] (\x cm,0) circle (.3cm);
    \draw[thick,fill=white!70] (4 cm, 2 cm) circle (.3cm);
    \draw[thick,fill=black!70] (10 cm,0) circle (.3cm);
    \foreach \y in {0.3, 2.3, 4.3, 6.3,8.3}
    \draw[thick] (\y cm,0) -- +(1.4 cm,0);
    \draw[thick] (4 cm,.3 cm) -- +(0,1.4 cm);
    \draw (0,.8) node {$\scriptscriptstyle{1}$};
    \draw (2,.8) node {$\scriptscriptstyle{2}$};
    \draw (4.5,.8) node {$\scriptscriptstyle{3}$};
    \draw (6,.8) node {$\scriptscriptstyle{4}$};
    \draw (8,.8) node {$\scriptscriptstyle{5}$};
    \draw (10,.8) node {$\scriptscriptstyle{6}$};
    \draw (4.5,2.8) node {$\scriptscriptstyle{7}$};
    
    \draw (15,0) node[anchor=east]  {$\tilde{E}_{7}$};
    \foreach \x in {18,20,22,24,26}
    \draw[thick,fill=white!70] (\x cm,0) circle (.3cm);
    \draw[thick,fill=white!70] (22 cm, 2 cm) circle (.3cm);
    \draw[thick,fill=black!70] (16 cm,0) circle (.3cm);
    \draw[thick,fill=black!70] (28 cm,0) circle (.3cm);
    \foreach \y in {16.3, 18.3, 20.3, 22.3,24.3,26.3}
    \draw[thick] (\y cm,0) -- +(1.4 cm,0);
    \draw[thick] (22 cm,.3 cm) -- +(0,1.4 cm);
    \draw (16,.8) node {$\scriptscriptstyle{0}$};
    \draw (18,.8) node {$\scriptscriptstyle{1}$};
    \draw (20,.8) node {$\scriptscriptstyle{2}$};
    \draw (22.5,.8) node {$\scriptscriptstyle{3}$};
    \draw (24,.8) node {$\scriptscriptstyle{4}$};
    \draw (26,.8) node {$\scriptscriptstyle{5}$};
    \draw (28,.8) node {$\scriptscriptstyle{6}$};
    \draw (22.5,2.8) node {$\scriptscriptstyle{7}$};
  \end{tikzpicture}

    \caption{Finite type Dynkin diagrams with cominuscule simple root marked in black (left
        column), and the corresponding affine Dynkin diagrams with both the cominuscule and
        the additional affine root marked in black (right column).}
    \label{TBL:comin}
\end{table}
        
 \begin{table}
  \begin{tikzpicture}[scale=.4]
    \draw (-1,0) node[anchor=east]  {$C_{n}$};
    \foreach \x in {2,4,6,8}
    \draw[thick,fill=white!70] (\x cm,0) circle (.3cm);
    \draw[thick,fill=black!70] (0 cm,0) circle (.3cm);
    \draw[thick] (0.3 cm,0) -- +(1.4 cm,0);
    \draw[dotted,thick] (2.3 cm,0) -- +(1.4 cm,0);
    \draw[thick] (4.3 cm,0) -- +(1.4 cm,0);
    \draw[thick] (6.3 cm, .1 cm) -- +(1.4 cm,0);
    \draw[thick] (6.3 cm, -.1 cm) -- +(1.4 cm,0);
    \draw[thick] (6.9 cm, 0 cm) -- +(.3 cm, .3 cm);
    \draw[thick] (6.9 cm, 0 cm) -- +(.3 cm, -.3 cm);
    \draw (0,.8) node {$\scriptscriptstyle{1}$};
    \draw (2,.8) node {$\scriptscriptstyle{2}$};
    \draw (4,.8) node {$\scriptscriptstyle{n-2}$};
    \draw (6,.8) node {$\scriptscriptstyle{n-1}$};
    \draw (8,.8) node {$\scriptscriptstyle{n}$};
    
    \draw (15,0) node[anchor=east]  {${A}^{(2)}_{2n-1}$};
    \foreach \x in {18,20,22,24}
    \draw[thick,fill=white!70] (\x cm,0) circle (.3cm);
    \draw[xshift=15 cm,thick,fill=black!70] (30: 17 mm) circle (.3cm);
    \draw[xshift=15 cm,thick,fill=black!70] (-30: 17 mm) circle (.3cm);
    \draw[dotted,thick] (18.3 cm,0) -- +(1.4 cm,0);
    \foreach \y in {20.3}
    \draw[thick] (\y cm,0) -- +(1.4 cm,0);
    \draw[thick] (22.3 cm, .1 cm) -- +(1.4 cm,0);
    \draw[thick] (22.3 cm, -.1 cm) -- +(1.4 cm,0);
    \draw[thick] (22.9 cm, 0 cm) -- +(.3 cm, .3 cm);
    \draw[thick] (22.9 cm, 0 cm) -- +(.3 cm, -.3 cm);
    \draw[xshift=17.5 cm,thick] (30: 3 mm) -- (136: 10.5 mm);
    \draw[xshift=17.5 cm,thick] (-30: 3 mm) -- (-136: 10.5 mm);
    \draw (16.5,1.5) node {$\scriptscriptstyle{1}$};
    \draw (16.5,-.2) node {$\scriptscriptstyle{0}$};
    \draw (18,.8) node {$\scriptscriptstyle{2}$};
    \draw (20,.8) node {$\scriptscriptstyle{n-2}$};
    \draw (22,.8) node {$\scriptscriptstyle{n-1}$};
    \draw (24,.8) node {$\scriptscriptstyle{n}$};
  \end{tikzpicture}
  
    \vspace{3mm}

  \begin{tikzpicture}[scale=.4]
    \draw (-1,0) node[anchor=east]  {$B_{n}$};
    \foreach \x in {0,2,4,6}
    \draw[thick,fill=white!70] (\x cm,0) circle (.3cm);
    \draw[thick,fill=black!70] (8 cm,0) circle (.3cm);
    \draw[thick] (0.3 cm,0) -- +(1.4 cm,0);
    \draw[dotted,thick] (2.3 cm,0) -- +(1.4 cm,0);
    \draw[thick] (4.3 cm,0) -- +(1.4 cm,0);
    \draw[thick] (6.3 cm, .1 cm) -- +(1.4 cm,0);
    \draw[thick] (6.3 cm, -.1 cm) -- +(1.4 cm,0);
    \draw[thick] (6.9 cm, .3 cm) -- +(.3 cm, -.3 cm);
    \draw[thick] (6.9 cm, -.3 cm) -- +(.3 cm, .3 cm);
    \draw (0,.8) node {$\scriptscriptstyle{1}$};
    \draw (2,.8) node {$\scriptscriptstyle{2}$};
    \draw (4,.8) node {$\scriptscriptstyle{n-2}$};
    \draw (6,.8) node {$\scriptscriptstyle{n-1}$};
    \draw (8,.8) node {$\scriptscriptstyle{n}$};
    
    \draw (15,0) node[anchor=east]  {${D}^{(2)}_{n+1}$};
    \foreach \x in {16,18,20,22,24,26}
    \draw[thick,fill=white!70] (\x cm,0) circle (.3cm);
    \foreach \x in {16,26}
    \draw[thick,fill=black!70] (\x cm,0) circle (.3cm);
    \draw[dotted,thick] (20.3 cm,0) -- +(1.4 cm,0);
    \foreach \y in {18.3,22.3}
    \draw[thick] (\y cm,0) -- +(1.4 cm,0);
    \draw[thick] (24.3 cm, .1 cm) -- +(1.4 cm,0);
    \draw[thick] (24.3 cm, -.1 cm) -- +(1.4 cm,0);
    \draw[thick] (24.9 cm, .3 cm) -- +(.3 cm, -.3 cm);
    \draw[thick] (24.9 cm, -.3 cm) -- +(.3 cm, .3 cm);
    \draw[thick] (16.3 cm, .1 cm) -- +(1.4 cm,0);
    \draw[thick] (16.3 cm, -.1 cm) -- +(1.4 cm,0);
    \draw[thick] (17.2 cm, .3 cm) -- +(-.3 cm, -.3 cm);
    \draw[thick] (17.2 cm, -.3 cm) -- +(-.3 cm, .3 cm);
    \draw (16,.8) node {$\scriptscriptstyle{0}$};
    \draw (18,.8) node {$\scriptscriptstyle{1}$};
    \draw (20,.8) node {$\scriptscriptstyle{2}$};
    \draw (22,.8) node {$\scriptscriptstyle{n-2}$};
    \draw (24,.8) node {$\scriptscriptstyle{n-1}$};
    \draw (26,.8) node {$\scriptscriptstyle{n}$};
  \end{tikzpicture}
  \caption{Finite type Dynkin diagrams with minuscule simple root marked in black (left
        column), and the corresponding twisted affine Dynkin diagrams with both the minuscule and
        the additional affine root marked in black (right column).  We use Kac's notation for the twisted
        affine Dynkin diagrams.  Note that in Dynkin's notation, ${A}^{(2)}_{2n-1}$ is denoted
        $\tilde{B}^t_n$, and ${D}^{(2)}_{n+1}$ is denoted $\tilde{C}^{t}_n$.}
    \label{TBL:min}
\end{table}

\end{document}